\documentclass[12pt]{amsart}
\usepackage{color}
\usepackage{amssymb}
\usepackage{comment}
\usepackage{pdfpages}
\usepackage{graphicx}
\usepackage{dcolumn}
\usepackage{tikz}
\usetikzlibrary{positioning,calc,intersections,shapes, snakes}

\RequirePackage[numbers]{natbib}
\RequirePackage[colorlinks=true, pdfstartview=FitV, linkcolor=blue,
  citecolor=blue, urlcolor=blue]{hyperref}
\RequirePackage{hypernat}
\usepackage{paralist}

\usepackage{graphicx}
\graphicspath{{./figures/}}
\usepackage{amsfonts}
\usepackage{amsmath}
\usepackage{amsthm}
\usepackage{amssymb}
\usepackage{amsbsy}
\usepackage{epsfig}
\usepackage{fullpage}
\usepackage{natbib, mathrsfs} 
\usepackage{verbatim}
\usepackage[latin1]{inputenc}
\usepackage{mhequ}
\usepackage{algorithm}
\usepackage{algorithmic}

\numberwithin{equation}{section}

\def \be{\begin{equs}}
\def \ee{\end{equs}}

\def \be{\begin{equs}}
\def \ee{\end{equs}}

\def \w{w}
\def \C{\rm{Cauchy(0, 1)}}

\def \R {\mathbb{R}}

\newtheorem{theorem}{Theorem}[section]
\newtheorem{lemma}[theorem]{Lemma}

\theoremstyle{plain}

\newtheorem*{thm-non}{Theorem}

\theoremstyle{definition}
\newtheorem{example}[theorem]{Example}

\newtheorem{remark}[theorem]{Remark}

\begin{document}

\title[Ratios and Cauchy Distribution]
{Ratios  and  Cauchy Distribution}


\author{Natesh S. Pillai$^{\ddag}$}
\thanks{$^{\ddag}$pillai@fas.harvard.edu, 
   Department of Statistics
    Harvard University, 1 Oxford Street, Cambridge
    MA 02138, USA}

 
\begin{abstract}
It is well known that the ratio of two independent standard Gaussian random variables follows a Cauchy distribution. Any convex combination of independent standard Cauchy random variables also follows a Cauchy 
distribution. In a recent joint work \cite{Pillai2016}, the author proved a surprising multivariate generalization of the above facts.
Fix $m > 1$ and let $\Sigma$ be a $m\times m$ positive semi-definite matrix. Let   $X,Y \sim \mathrm{N}(0,\Sigma)$ be independent vectors. Let $\vec{\w}=(\w_1, \dots, \w_m)$ be a vector of non-negative numbers with
$\sum_{j=1}^m \w_j = 1.$
It was conjectured in \cite{DrtonXiao16}, and proved in \cite{Pillai2016}, that the random variable
\be 
Z = \sum_{j=1}^m \w_j {X_j \over Y_j}\;
\ee
also has the standard Cauchy distribution.
In this note, we provide some more understanding of this result and give a number of natural generalizations. In particular, we observe that if $(X,Y)$ have the same marginal distribution, they need neither be independent nor be jointly normal for $Z$ to be Cauchy distributed. In fact, our calculations suggest that joint normality of $(X,Y)$ may be the only instance in which they can be independent.
Our results also give a method to construct copulas of Cauchy distributions.
\end{abstract}
\maketitle




\section{Introduction}
Fix $m \in \mathbb{N}$ and let $\Sigma$ be a $m \times m$ positive semi-definite matrix. 
Let $X,Y \sim \mathrm{N}(0,\Sigma)$ be independent vectors. We denote the column vectors as
$X= (X_1, \dots, X_m)$ and $Y = (Y_1, \dots,Y_m)$.
Let $\vec{\w}=(\w_1, \dots, \w_m)$ be such that 
\be \label{eqn:pi}
 \sum_{j=1}^m \w_j = 1, \quad  \w_j\ge 0, \,j=1,\ldots,  m.
\ee
Throughout the paper, the vector $\vec{\w}$ will be assumed to be deterministic, but all of our results hold if $\vec{\w}$ is random but independent of $(X,Y)$.
It was conjectured in \cite{DrtonXiao16}, and recently proved in \cite{Pillai2016}, that the random variable
\be \label{eqn:Z}
Z = \sum_{j=1}^m \w_j {X_j \over Y_j}\;
\ee
has the standard Cauchy distribution with probability density
\be \label{eqn:Cauchyden}
f_{Z}(z) = {1 \over \pi^2} {1 \over 1+ z^2}.
\ee 
 This result is quite surprising and has many important applications, especially in determining the asymptotic behavior of Wald tests in factor models, graphical models, contigency tables, \textit{etc.} For instance, one important consequnce of this result is that, see Theorem 2.2 of \cite{Pillai2016}, if $X \sim \mathrm{N}(0,\Sigma)$, then
 \be \label{eq:levy}
\Big({\w_1 \over X_1}, \dots, {\w_m \over X_m}\Big)^\top \Sigma
\Big({\w_1 \over X_1},  \dots, {\w_m \over X_m}\Big) \sim  \chi^{-2}_1
\ee
where $\chi^{-2}_1$ denotes the inverse chi-squared variable with $1$ degree of freedom. Thus the quadratic form in \eqref{eq:levy} is a pivotal quantity for $\Sigma$ and is a natural test statistic.
 See \cite{DrtonXiao16} for an extensive list of applications and further discussion. 

The proof in \cite{Pillai2016} is short and uses a geometric characterization of the Cauchy distribution. Nevertheless, the result still seems mysterious. Inspection of the proof in \cite{Pillai2016} reveals that it holds in much greater generality. In this note, we provide some more understanding of this result and give a number of natural generalizations. In particular, we relax the assumptions that $(X,Y)$ are independent and that they are jointly normal. \par The following question was posed in \cite{Pillai2016}:  ``for a  given family of random variables $Z_1,\dots,Z_m$, can the dependence among them be overwhelmed
by the heaviness of their marginal tails (\textit{e.g.}, $Z_j = {X_j \over Y_j}$) in
determining the stochastic behavior of their linear combinations?''
The main result of this paper gives numerous examples that answer the above question in the affirmative.  An interesting direction for further inquiry is to fully characterize this phenomenon.
\section{Rotational Invariance and Cauchy}
We will write $Z \sim \C$ to denote that the random variable $Z$ has
the Cauchy distribution with location parameter $0$ and scale parameter $1$, with density $f_Z(z)$ as in Equation \eqref{eqn:Cauchyden}.
Our key observation starts with the following fact:
if $\Theta \sim \mathrm{Unif}(-\pi,\pi]$, then 
\be \label{eqn:tanT}
\tan(\Theta) \sim \C.
\ee
Let $X, Y \sim \mathrm{N}(0,1)$ be independent. 
Then, $Z = {X \over Y} \sim \C$. This is easy to see using \eqref{eqn:tanT}.  Write $(X,Y) = (R \sin(\Theta), R \cos(\Theta))$ with $R \in (0, \infty)$ and
$\Theta \in (\pi,\pi]$. Thus, $Z = \tan(\Theta)$.
Due to the rotational invariance of the joint density of $(X,Y) \in \mathbb{R}^2$, $\Theta \sim \mathrm{Unif}(-\pi,\pi]$. Thus from \eqref{eqn:tanT}, it follows that $ Z \sim \C$. This argument did not use the fact that $(X,Y)$ are jointly Gaussian or independent, but rather that their joint distribution is rotationally invariant in $\mathbb{R}^2$.\par
The above reasoning thus applies to all other rotationally invariant joint distributions for $(X,Y)$. For instance,
if the pair $(X,Y)$ have joint densities
\be
f_{X,Y}(x,y) \propto {1 \over (1+ x^2+y^2)^n}, \quad n \geq 1
\ee
or 
\be
g_{X,Y}(x,y) \propto (x^2+y^2) \exp\{-\sqrt{x^2 + y^2}\},
\ee
then $Z = {X \over Y} \sim \C$. This observation 
also generalizes to multivariate $X$ and $Y$ and is the content of Theorem \ref{thm:main} below. For multivariate $(X,Y)$, in addition to rotational invariance, there are many more ways of incorporating symmetry, or antisymmetry, in their joint density that will lead to a Cauchy 
distribution; see Remark \ref{rem:symm}.

As in \cite{Pillai2016}, the proof of our main result relies crucially on the following result from \cite{pitman1967}. Also see \cite{williams1969cauchy} and \cite{letac1977}) for additional discussions. Lemma \ref{thm:cauchywilres} is proved in \cite{pitman1967} using the Residue theorem. A geometric proof for $m=2$ can be found in \cite{Cohen12}. 
\begin{figure}
\centering
\begin{tikzpicture}[scale=0.85]
\draw (2,2) circle (2.0 cm);
\draw (-2,0) -- (9,0); 
\draw[dashed] (0,4) -- (4,4); 
\draw[dashed] (2,4.5) -- (2,0);
\draw (2,4) -- (5,0);
\draw (2,4) -- (7,0);
\draw [fill] (2,2) circle [radius = 0.05];
\draw [fill] (5,0) circle [radius = 0.05];
\draw [fill] (2,0) circle [radius = 0.05];
\draw [fill] (2,4) circle [radius = 0.05];
\draw [fill] (7,0) circle [radius = 0.05];
\draw [gray] (2,0) -- (2,-1);
\draw [gray] (7,0) -- (7,-1);
\coordinate (A) at (2,4);
\coordinate (X) at (2,0);
\coordinate (Y) at (5,0);
\coordinate (Z) at (7,0);
\draw [fill,blue] (6,0) circle [radius = 0.10];
\draw [->] (6,0) -- (7,2); 
\begin{scope}
\path[clip] (A) -- (X) -- (Y);
\fill[red, opacity=0.5, draw=black] (A) circle (5mm);
\draw (2.3, 3.2) node {$\Theta$};
\end{scope}
\begin{scope}
\path[clip] (Y) -- (A) -- (Z);
\fill[red, opacity=0.5, draw=black] (A) circle (9mm);
\draw (3.3, 2.7) node {$U$};
\end{scope}
\draw [<->] (2,-0.25) -- (5.0,-0.25);
\draw [<->] (2,-1) -- (7.0,-1);
\draw (5,0.35) node{P};
\draw (7,0.35) node{Q};
\draw (3.5,-0.5) node{$Z = \tan(\Theta)$};
\draw (4.5,-1.4) node{$Z' = \tan(\Theta + U)$};
\draw (7.0,2.3) node{$w_1 \tan(\Theta) + w_2\tan(\Theta + U)$};
\end{tikzpicture}
\caption{The circle has diameter $1$. If $\Theta$ and $U$ are independent and $\Theta \sim \mathrm{Unif}(-\pi,\pi]$, then both $Z= \tan(\Theta)$ and $Z'= \tan(\Theta + U)$ are Cauchy distributed. Lemma \ref{thm:cauchywilres} yields that 
any point in the line segment $PQ$ is also Cauchy distributed marginally.}
 \label{fig:geom}
\end{figure}
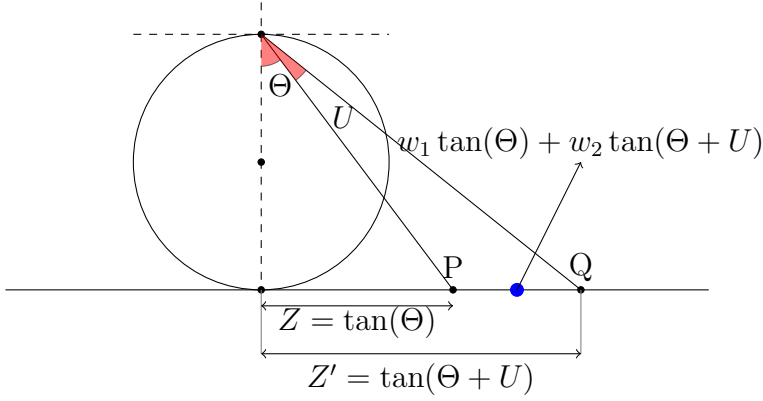
\begin{lemma} \label{thm:cauchywilres}
Let $\Theta_1 \sim \mathrm{Unif}(-\pi,\pi]$,  and $\{w_1,\ldots, w_m\}$ be independent of $\Theta_1$, where $w_j \ge 0$ and $\sum_j w_j=1.$ \ Then for  any $\{u_1, \ldots, u_{m}\}$, where $u_j \in \mathbb{R}$,
\be
\sum_{j=1}^m \w_j\tan(\Theta_1 + u_j) \sim \C.
\ee
\end{lemma}
Figure \ref{fig:geom} gives a geometric interpretation of Lemma \ref{thm:cauchywilres} for $m=2$.
\section{Cauchy from convex combination of dependent ratios}
Consider a symmetrix matrix $F \in \R^{2m \times 2m}$ of the following form:
\be \label{eqn:Fmat}
F = \left( \begin{array}{cc}
A & B  \\
-B & A  \end{array} \right)
\ee
where $A$ is an arbitrary symmetric $m \times m$ matrix and $B$ 
is an arbitrary $m \times m$ antisymmetric matrix.
The following is our main result.
\begin{theorem} \label{thm:main}
Fix $m >1$. Let $X, Y \in \R^m$ be vectors with joint density
\be
f_{X,Y}(x,y) = K \prod_{i=1}^n h_i \big((X^\top,  Y^\top) F_i (X,  Y)\big)
\ee
where $n \in \mathbb{N}$, $h_i : \R \mapsto \R_+$ are arbitrary measurable functions, $F_i$ are matrices of the form \eqref{eqn:Fmat} and $K$ is the normalizing constant. 
Then, for any vector $w$ satisfying \eqref{eqn:pi},  $Z = \sum_{j=1}^m \w_j {X_j \over Y_j} \sim \C$. Furthermore, for any $1 \leq j \leq m$,
$Z_j = {X_j \over Y_j} \sim \C$. 
\end{theorem} 

\begin{proof}
The proof is almost identical to that of Theorem 1.1 of \cite{Pillai2016}.
Let $F_i$ be the matrix
\be 
F_i = \left( \begin{array}{cc}
A_i & B_i  \\
-B_i & A _i \end{array} \right)
\ee
where $A_i \in \mathbb{R}^{m\times m}$ is symmetric and $B_i \in \mathbb{R}^{m\times m}$ is antisymmetric.\par
Set $(X_j,Y_{j}) = (R_{j}\sin(\Theta_j),R_j\cos(\Theta_j))$, where
$0 \leq R_j < \infty$ and $\Theta_j \in (-\pi, \pi]$.  We write $R=\{R_1,\ldots, R_m\}$ and $\Theta=\{\Theta_1, \ldots,\Theta_m\}$. The Jacobian for the transformation $(X,Y) \mapsto (R,\Theta)$ is $\prod_{j=1}^m R_j$. 
The joint density of $(R,\Theta)$  is
\be  
f_{R, \Theta}&(r, \theta) \propto \\
 &\hspace{-0.5cm}\prod_{j=1}^m r_j\, \prod_{i=1}^n h_i\Big(\sum_{j=1}^m (A_i)_{jj} \, r_j^2 + 2r_j r_k\sum_{k>j} \big((A_i)_{jk}\,\cos(\theta_j - \theta_k)+ (B_i)_{jk}\,\sin(\theta_j - \theta_k)\big)\Big) \label{eqn:J1}
\ee
for $r \in [0,\infty)^m$ and $\theta \in (-\pi, \pi]^m$.
 
We then  make a further  transformation, $\mathcal{F}:(-\pi,\pi]^m \mapsto (-\pi,\pi]^m$, with 
$\mathcal{F}(\Theta_1, \dots, \Theta_m)=  (\Theta_1, U_2, \dots, U_m)$, where
\be \label{eq:mod} 
U_j= (\Theta_j - \Theta_1)+ 2\pi [\mathbf{1}_{\{\Theta_j - \Theta_1 \leq -\pi\}} -\mathbf{1}_{\{\Theta_j - \Theta_1 >\pi\}  }], \quad 2\le  j \leq m.
\ee 
This is a form of $U_j=(\Theta_j - \Theta_1)\mod(2\pi)$, but with the assurance  that the support of $U_j$ is $(-\pi, \pi]$ regardless of the value of $\Theta_1$, and that $U_j-U_k=(\Theta_j-\Theta_k)\mod(2\pi)$, and $\Theta_j=(\Theta_1+U_j)\mod(2\pi)$.
The map $\mathcal{F}$ is one-to-one as shown in Figure 1 of \cite{Pillai2016}.
 Furthermore, the points where the map $\mathcal{F}$ is \emph{not} differentiable is contained in the set
 \be
\big \{\Theta \in (-\pi,\pi]^m: \Theta_j - \Theta_1 \in \{-\pi,\pi\} \, \mathrm{for \,\,some\,\,} j \geq 2 \big\}.
 \ee 
Clearly this set has Lebesgue measure zero. Outside this set, we have ${\partial U_j \over \partial \Theta_j} = 1$. Thus the Jacobian of the map $\mathcal{F}$  is 1 for all $\Theta \in (-\pi,\pi]^m$ except for the above measure zero set. \par
Set $U_1 \equiv 0$ and denote $U = (U_1, \dots, U_m)$.
Since $\cos(W_1) = \cos(W_2)$ and $\sin(W_1) = \sin(W_2)$ for any $W_1=W_2\mod(2\pi)$,
we can write the joint density in the new coordinates as 
\be 
 f_{R, \Theta_1, U}&(r, \theta_1, u) \propto \\ 
 &\hspace{-0.5cm}\prod_{j=1}^m r_j\, \prod_{i=1}^n h_i\Big(\sum_{j=1}^m (A_i)_{jj} \, r_j^2 + 2r_j r_k\sum_{k>j} \big((A_i)_{jk}\,\cos(u_k - u_j)+ (B_i)_{jk}\,\sin(u_k - u_j)\big)\Big)
\ee
with $r \in [0,\infty)^m, \theta_1 \in (-\pi, \pi], u_1 = 0$ and $u_2, \dots, u_m \in (-\pi, \pi]$.
The only observations we need from the above line are: (i) $\Theta_1$ is independent of $U$ and  (ii) $\Theta_1 \sim \mathrm{Unif}(-\pi,\pi]$. 
But   $Z= \sum_{j} w_j {X_j \over Y_j}$ can be written as 
\be
Z = \sum_{j=1}^m \w_j {X_j \over Y_j}= \sum_{j=1}^m \w_j \tan(\Theta_j ) 
= \sum_{j=1}^m \w_j \tan(\Theta_1 + U_j), 
\ee 
because $\tan(W_1) = \tan(W_2)$ for any $W_1=W_2\mod(2\pi)$. Since $U$ is independent of $\Theta_1$, conditional on  $U$, Lemma  \ref{thm:cauchywilres} yields that $Z\sim \C$.  It follows immediately that $Z$ is also marginally distributed as $\C$.\par
Since $\Theta_1 \sim \mathrm{Unif}(-\pi,\pi)$, by \eqref{eqn:tanT} it follows that 
$Z_1 = {X_1 \over Y_1} = \tan(\Theta_1) \sim \C$. Since the ordering of variables in the preceding argument was arbitrary, by symmetry it follows immediately that $\Theta_j \sim \mathrm{Unif}(-\pi,\pi]$ and thus $Z_j = {X_j \over Y_j} = \tan(\Theta_j) \sim \C$.
 \end{proof}
 \begin{remark}\label{rem:xian}
 If the joint density $f_{X,Y}(x,y)$ can be written as a mixture of Gaussians:
 \be[eqn:gausmix]
 f_{X,Y}(x,y) &= \sum_{n=1}^\infty \alpha_n g^{(n)}_{X,Y}(x,y) \\
 g^{(n)}_{X,Y}(x,y) &\propto \exp \big \{ -{ 1 \over 2} (X^\top \Sigma^{-1}_n X + Y^\top \Sigma^{-1}_n Y) \big \}
 \ee
 with $\sum_n \alpha_n = 1$ and $\Sigma_n$ are arbitrary positive definite matrices, then the main result of \cite{Pillai2016} will immediately yield that 
 $Z  = \sum_{j=1}^m w_j {X_j \over Y_j} \sim \C$. To see this, let $N \in \mathbb{N}$ be a discrete random variable with $\mathbb{P}(N=n) = \alpha_n$. Conditional on $N=n$, let $(X,Y)$ has joint density $g^{(n)}_{X,Y}(x,y)$. Then, marginalizing over $N$, we get that the joint density of $(X,Y)$ is $f_{X,Y}$. Now, for each $n \in \mathbb{N}$, the main result of \cite{Pillai2016} will yield that if $(X,Y)$ has joint density $g^{(n)}_{X,Y}(x,y)$, then $Z \sim \C$. Averaging over $N$ yields the claim. In \cite{kelker1970}, the author notes that certain families of spherically symmetric distributions can be expressed as \eqref{eqn:gausmix}. The above argument will then yield that $Z \sim \C$.
 Theorem \ref{thm:main} generalizes this observation further in two ways. First, Theorem \ref{thm:main} shows that if the joint density of $(X,Y)$ is proportional to product of spherically symmetric densities, then $Z \sim \C$. Second, 
it shows how to incorporate antisymmetry.
 
 \end{remark}
\begin{example} \label{eqn:cop}
Let $\Sigma$ be a $2m \times 2m$ positive definite matrix and $h(x) = e^{-{1 \over 2} x}$. 
Set
\be
f_{X,Y}(x,y) \propto h(X^\top \Sigma^{-1} X + Y^\top \Sigma^{-1} Y) =
\exp \big \{ -{ 1 \over 2} (X^\top \Sigma^{-1} X + Y^\top \Sigma^{-1} Y) \big \}
\ee
so that $(X,Y)$ are jointly Gaussian but independent. Theorem \ref{thm:main} yields that $Z  = \sum_{j=1}^m w_j {X_j \over Y_j} \sim \C$. This result was of course conjectured in \cite{DrtonXiao16} and proved in \cite{Pillai2016}. Interestingly, to our knowledge, this is the only example that satisfies the hypothesis of Theorem \ref{thm:main} such that $(X,Y)$ are independent. A natural generalization of this density that satisfies the hypothesis of Theorem \ref{thm:main} is
\be \label{eqn:natgen}
f_{X,Y}(x,y) \propto
(X^\top A X + Y^\top A Y)^{2q}\exp \big \{ -{ 1 \over 2} (X^\top \Sigma^{-1} X + Y^\top \Sigma^{-1} Y) \big \}
\ee
where $A$ is an arbitrary $m \times m$ symmetric matrix and $q \in \mathbb{N}$. See equation \eqref{eqn:copula} for an example of a density of the form \eqref{eqn:natgen}.
\end{example}
\begin{example}\label{eg:Bmat}
Take $m=2$ and set 
\be
f_{X,Y}(x,y) \propto  (x_1y_2 - x_2y_1)^2 \exp\big\{-{1\over 2}(x_1^2  + x_2^2 + y_1^2 + y_2^2)\big\}.
\ee
The joint density $f_{X,Y}$ does satisfy the hypothesis of Theorem \ref{thm:main},
and thus $Z = \sum_{j=1}^2 w_j {X_j \over Y_j} \sim \C$. 
\end{example}
\begin{example}\label{eg:Bmat2}
Take $m=2$. Consider a positive definite matrix $F$ of the form
\be
F = \left( \begin{array}{cccc}
a & c &0&d \\
c& b &-d&0 \\
0&-d&a&c\\
d&0&c&b \end{array} \right)
\ee
where $a,b,c,d \in \mathbb{R}$. If $\min(a,b) > |c| + |d|$, then $F$ will be diagonally dominant and thus positive definite.
Let $(X,Y)$ have a joint Gaussian distribution with precision matrix $F$, \textit{i.e.,}
\be  \label{eqn:densityGaussB}
f_{X,Y}(x,y) \propto \exp\{-{1 \over 2} (x^\top y^\top) F (x, y)\}.
\ee
The density $f_{X,Y}$ above satisfies the hypothesis of Theorem \ref{thm:main} and thus  $Z = \sum_{j=1}^2 w_j {X_j \over Y_j} \sim \C$. It is well known that zeroes in the precision matrix indicate conditional independence. Thus if $d\neq 0$ in the joint density $f_{X,Y}$ in \eqref{eqn:densityGaussB}, then $X_1$ is conditionally independent of $Y_1$ given $(X_2,Y_2)$. Similarly, $X_2$ is conditionally independent of $Y_2$ given $(X_1, Y_1)$.
\end{example}
Thus Example \ref{eg:Bmat2} shows that even in the Gaussian case, independence of $(X,Y)$ is not needed; certain conditional independence relations might suffice. 
\begin{example} \label{eg:Bmat3}
Taking $c=0$ and $d \neq 0$ in the matrix $F$ in Example
\ref{eg:Bmat2} reveals the following surprise. 
Pick $\rho \in (-1,1)$ and set
\be
\Sigma_\rho = \left( \begin{array}{cc}
1 & \rho  \\
\rho & 1 \end{array} \right), \quad \Sigma_{-\rho} = \left( \begin{array}{cc}
1 & -\rho  \\
-\rho & 1 \end{array} \right).
\ee
Let 
$(X_1,Y_2) \sim N(0, \Sigma_\rho)$ and $(X_2,Y_1) \sim N(0,\Sigma_{-\rho})$.
Let $(X_1,Y_2)$ be independent of $(X_2, Y_1)$. Thus they have the joint density
 \be 
f_{X,Y}(x,y) \propto
 \exp \Big \{-{1 \over 2(1-\rho^2)} ((x_1^2 + y_2^2) - 2\rho x_1y_2 ) \Big \} \exp \Big \{-{1 \over 2} ((x_2^2 + y_1^2)+ 2\rho x_2y_1 ) \Big\}.
\ee
This corresponds to the joint density $f_{X,Y}$ in \eqref{eqn:densityGaussB} in Example
\eqref{eg:Bmat} with values $c=0$, $d = {-\rho/(1- \rho^2)}$ and $a = b = {1/(1-\rho^2)}$ for the entries of matrix $F$. Thus, it follows that the result conjectured in \cite{DrtonXiao16} also holds with the pairs $(X_1, Y_2)$ and $(X_2,Y_1)$ with $\mathrm{cov}(X_1,Y_2) = -\mathrm{cov}(X_2,Y_1)$! Generalization of this example to $m > 2$ will be of interest.
\end{example}
\begin{remark} \label{rem:symm}
Examples \ref{eg:Bmat}--\ref{eg:Bmat3} show that rotational invariance is not the key to full generality. Theorem \ref{thm:main} can be generalized further by only requiring that the joint density of $(X,Y)$ in polar coordinates depends on $\Theta$ only via $2\pi$-periodic functions of $(\theta_j - \theta_k)$ for $1\leq j \neq k \leq m$. 
The author refrained from doing so, to keep the exposition simple. The author does not know if this formulation might fully characterize the family of joint distributions for $(X,Y)$ so as to have $Z = \sum_j w_j {X_j \over Y_j} \sim \C$. 
\end{remark}
 \section{Copulas of Cauchy Distributions} 
 Theorem \ref{thm:main} also yields that the marginal distributions of the ratios $Z_j = {X_j \over Y_j} \sim \C$. This gives a natural way of constructing  copulas of Cauchy distributions. In this section we work out the simplest case for $m=2$. Our calculations yield a novel and interesting family of bivariate copulas with Cauchy marginals.
Let 
\be
\Sigma =  \left( \begin{array}{cc}
\rho & 0  \\
0 & \rho  \end{array} \right), \quad \rho \in (-1,1).
\ee
Let $X,Y \sim \mathrm{N}(0,\Sigma)$ be independent. Let $w_1, w_2 \geq 0$ with $w_1 +w_2 = 1$. Let
\be
Z =  w_1 Z_1+ w_2 Z_2
\ee
where $Z_1 =  {X_1 \over Y_1}, \, Z_2 = {X_2 \over Y_2}.$
Theorem \ref{thm:main} implies that $Z \sim \C$.
 Let $f^\rho_{Z_1,Z_2}(z_1,z_2)$ denote the joint distribution of $(Z_1,Z_2)$.  
\begin{lemma} \label{lem:frhoC}
The joint density $f^\rho_{Z_1,Z_2}$ is an infinite mixture of bivariate copulas of Cauchy densities:
\be [eqn:fz1z2jointexp]
f^\rho_{Z_1,Z_2}(z_1,z_2) 
&=
(1-\rho^2)\sum_{n=0}^\infty \rho^{2n} \,f^{(n)}_{Z_1,Z_2}(z_1,z_2) \\
  f^{(n)}_{Z_1,Z_2}(z_1,z_2)(z_1,z_2) &= {2^{2n} \over {2n \choose n}}{1\over \pi^2} {(1+z_1z_2)^{2n} \over (1+z_1^2)^{n+1} (1+z_2^2)^{n+1} }.
\ee
Moreover, for every $n \in \mathbb{N}$, if $(C_1,C_2) \sim f^{(n)}_{Z_1,Z_2}$, then $C_1,C_2 \sim \C$ and $w_1 C_1 + w_2 C_2 \sim \C$.
\end{lemma}
\begin{proof} We make the following transformation:
\be
 V_1 = Y_1, \, V_2 = Y_2.
\ee
The Jacobian of the transformation $(X_1,X_2, Y_1,Y_2) \mapsto (Z_1,Z_2,V_1,V_2)$
is $|V_1V_2|$. The joint density of $(Z_1, Z_2, V_1, V_2)$ is 
\be 
f_{Z_1,Z_2, V_1,V_2}&(z_1,z_2,v_1,v_2)  \label{eqn:jointdis} \\
&\hspace{-2cm} ={1 \over (2\pi)^2 (1-\rho^2)}|v_1v_2|\exp\Big\{ -{1 \over 2(1-\rho^2)}\Big(v_1^2(1+ z_1^2) + v_2^2(1+z^2_2) - 2\rho v_1v_2 (1+z_1 z_2)\Big)\Big\}. 
\ee
Taylor expansion of \eqref{eqn:jointdis} yields 
\be \label{eqn:fz1z2exp}
f_{Z_1,Z_2, V_1,V_2}(z_1,z_2,v_1,v_2) &  \\
&\hspace{-3cm} ={1 \over (2\pi)^2 (1-\rho^2)}|v_1v_2|\exp\Big\{ -{1 \over 2(1-\rho^2)}\Big(v_1^2(1+ z_1^2) + v_2^2(1+z^2_2)\Big)\Big\} \\
&\hspace{3cm}\Big[\sum_{n=0}^\infty {1 \over n!} { \rho^n \over (1-\rho^2)^n} (v_1v_2)^n (1+z_1z_2)^n \Big].
\ee
Now using the fact for any $c>0$,
\be \label{eqn:wmarg}
\int_{\mathbb{R}} |v| v^n e^{-{1 \over 2} {v^2 \over c}(1+z^2)} dv &=
2^{n/2 +1} c^{n/2 + 1} {1 \over (1+ z^2)^{n/2 +1}} \Gamma({n \over 2} +1) \,1_{n \in 2\mathbb{Z}},
\ee
we can integrate over $v_1, v_2$ in equation \eqref{eqn:fz1z2exp} to get
\be
f^\rho_{Z_1,Z_2}(z_1,z_2) 
&=
(1-\rho^2)\sum_{n=0}^\infty \rho^{2n} {2^{2n} \over {2n \choose n}}{1\over \pi^2} {(1+z_1z_2)^{2n} \over (1+z_1^2)^{n+1} (1+z_2^2)^{n+1} }\\
&=(1-\rho^2)\sum_{n=0}^\infty \rho^{2n} \,f^{(n)}_{Z_1,Z_2}(z_1,z_2)
\ee
proving the first claim.\par
From equations \eqref{eqn:fz1z2exp} and \eqref{eqn:wmarg}, it can be seen that the random variables $C_1,C_2 \sim f^{(n)}_{Z_1,Z_2}$ can be generated by via the ratios $C_1 = {E_1 \over F_1}$, $C_2 = {E_2 \over F_2}$, where $(E,F)$ have the joint density,
\be \label{eqn:copula}
f_{E,F}(e,f) \propto (e_1e_2 + f_1f_2)^{2n} \exp
\Big \{{-1\over 2}(e_1^2 + e_2^2 + f_1^2 + f_2^2)\Big \}.
\ee 
The density $f_{E,F}$ satisfies the hypothesis of Theorem \ref{thm:main}; see Example \ref{eqn:cop}. Thus we have $C_1, C_2 \sim \C$ and $w_1 C_1 + w_2 C_2 \sim \C$,
and the proof is finished.
\end{proof}
The first term $f^{(0)}_{Z_1,Z_2}$ in the expansion \eqref{eqn:fz1z2jointexp}
is just the product of independent Cauchy densities:
\be
f^{(0)}_{Z_1,Z_2}(z_1,z_2) = {1 \over \pi^2} {1 \over (1+z_1^2)(1+z_2^2)}.
\ee
The role of the parameter $\rho$ in \eqref{eqn:fz1z2jointexp} is also interesting. It appears only as a weight in the mixture and neatly decouples from the probability densities $f^{(n)}$. It will be of interest to know if this phenomenon persists in higher dimensions $(m>2)$ as well. Finally, using the calculations in this section, it can be verified that $f^\rho_{Z_1,Z_2}$ has the closed form expression
\be
f^\rho_{Z_1,Z_2}(z_1,z_2)&=
{1-\rho^2 \over \pi^2} {1 \over (1 + z_1^2)(1+ z_2^2) - \rho^2 (1+z_1z_2)^2} \\ 
&\hspace{0.4cm}+{1-\rho^2 \over \pi^2}  {\rho(1+z_1z_2) \over {((1 + z_1^2)(1+ z_2^2) - \rho^2 (1+z_1z_2)^2)^{3/2}}} \sin^{-1} {\rho(1+z_1 z_2) \over \sqrt{1+z_1^2}\sqrt{1+z_2^2}}.
\ee
\section*{Acknowledgement}
The author is partially supported by an ONR grant. He wishes to thank various colleagues for their interest in this work.
Special thanks are due to Mathias Drton for introducing the author to this problem, Gerard Letac for corrections, Christian Robert for comments that led to Remark \ref{rem:xian} and Xiao-Li Meng for inspiration and constant encouragement.
\bibliographystyle{alpha}
\bibliography{chisq}

\begin{thebibliography}{Coh12}

\bibitem[Coh12]{Cohen12}
Michael~P Cohen.
\newblock Sample means of independent standard {C}auchy random variables are
  standard {C}auchy: a new approach.
\newblock {\em American Mathematical Monthly}, 119(3):240--244, 2012.

\bibitem[DX16]{DrtonXiao16}
M.~Drton and H.~Xiao.
\newblock Wald tests of singular hypotheses.
\newblock {\em Bernoulli}, 22(1):38--59, 2016.

\bibitem[Kel70]{kelker1970}
Douglas Kelker.
\newblock Distribution theory of spherical distributions and a location-scale
  parameter generalization.
\newblock {\em Sankhy{\=a}: The Indian Journal of Statistics, Series A}, pages
  419--430, 1970.

\bibitem[Let77]{letac1977}
G{\'e}rard Letac.
\newblock Which functions preserve {C}auchy laws?
\newblock {\em Proceedings of the American Mathematical Society},
  67(2):277--286, 1977.

\bibitem[PM16]{Pillai2016}
N.S. Pillai and X-L. Meng.
\newblock An unexpected encounter with {C}auchy and {L}\'evy.
\newblock {\em The Annals of Statistics}, 2016.
\newblock Forthcoming.

\bibitem[PW67]{pitman1967}
EJG Pitman and EJ~Williams.
\newblock Cauchy-distributed functions of {C}auchy variates.
\newblock {\em The Annals of Mathematical Statistics}, 38(3):916--918, 1967.

\bibitem[Wil69]{williams1969cauchy}
EJ~Williams.
\newblock Cauchy-distributed functions and a characterization of the {C}auchy
  distribution.
\newblock {\em The Annals of Mathematical Statistics}, 40(3):1083--1085, 1969.

\end{thebibliography}

\end{document}